\documentclass[final]{dmtcs-episciences}
\received{2016-5-12}
\revised{2017-7-18}
\accepted{2017-8-30}
\publicationdetails{19}{2017}{3}{4}{1485}

\usepackage[english]{babel}
\usepackage[utf8]{inputenc} 
\usepackage{amsfonts,amsmath,amssymb,amsthm}
\usepackage{wasysym}
\usepackage{graphicx,subcaption}
\usepackage{hyperref}
\usepackage[round]{natbib}

\newcommand{\ca}{\textsc{ca}}                 
\newcommand{\eg}{\textit{e.g.}}               
\newcommand{\ie}{\textit{i.e.}}               
\newcommand{\sft}{\textsc{sft}}               
\newcommand{\wrt}{with respect to }           

\newcommand{\Acal}{\ensuremath{\mathcal{A}}}  
\newcommand{\Forb}{\ensuremath{\mathcal{F}}}  
\newcommand{\Lang}{\ensuremath{\mathcal{L}}}  

\newcommand{\deux}{\ensuremath{\{0,1\}}}      
\newcommand{\Neigh}{\ensuremath{\mathcal{N}}} 
\newcommand{\PF}{\ensuremath{\mathcal{PF}}}   
\newcommand{\restrict}[2]{\ensuremath{        
    \left.{#1}\right|_{{#2}}
}}
\newcommand{\xor}{\ensuremath{\mathrm{xor}}}  
\newcommand{\GG}{\ensuremath{\mathbb{G}}}     
\newcommand{\gen}{\ensuremath{B}}             
\newcommand{\update}[1]{\ensuremath{
    \xrightarrow{#1}%
}}

\newcommand{\ifthenelsedef}[3]{ 
  \ensuremath{\left\{ \begin{array}{ll}
      {#2} & \mathrm{if} \; {#1} \\
      {#3} & \mathrm{otherwise}
    \end{array} \right.}
}

\newcommand{\Zset}{\integers}

\theoremstyle{plain}

\newtheorem{theorem}{Theorem}
\newtheorem{corollary}{Corollary}
\newtheorem{lemma}{Lemma}
\newtheorem{proposition}{Proposition}
\newtheorem{conjecture}{Conjecture}

\theoremstyle{definition}

\newtheorem{definition}{Definition}
\newtheorem{example}{Example}
\newcommand{\exampleqed}{\ensuremath{\ocircle}\par}

\title{Post-surjectivity and balancedness of cellular automata over groups}

\newcommand{\SC}{Silvio Capobianco}
\newcommand{\JK}{Jarkko Kari}
\newcommand{\ST}{Siamak Taati}

\newcommand{\TTU}{
  {Department of Software Science, School of Information Technology, Tallinn University of Technology. \newline \href{mailto:silvio@cs.ioc.ee}{\texttt{silvio@cs.ioc.ee}}, \href{mailto:silvio.capobianco@gmail.com} {\texttt{silvio.capobianco@gmail.com}}}
}

\newcommand{\UTU}{
  Department of Mathematics and Statistics, University of Turku.
  \href{mailto:jkari@utu.fi}{\texttt{jkari@utu.fi}}
}
\newcommand{\LEIDENUNIV}{
  Mathematical Institute, Leiden University.
  \href{mailto:siamak.taati@gmail.com}{\texttt{siamak.taati@gmail.com}}
}

\newcommand{\TNXSC}{ 
  This research was supported by
  the ERDF funded project Coinduction,
  the Estonian Ministry of Education and Research
  institutional research grant IUT33-13,
  and the Estonian Science Foundation grant no. 9398.
}
\newcommand{\TNXJK}{ 
  Research supported by the Academy of Finland grant 296018.
}
\newcommand{\TNXST}{ 
  The work of ST is supported by ERC Advanced Grant 267356-VARIS
  of Frank den Hollander.
}

\author{\SC \affiliationmark{1} \thanks{\TNXSC}
  \and \JK \affiliationmark{2} \thanks{\TNXJK}
  \and \ST \affiliationmark{3} \thanks{\TNXST}
}

\affiliation{
  \TTU \\
  \UTU \\
  \LEIDENUNIV
}


\begin{document}

\maketitle

\begin{abstract}
  We discuss cellular automata over arbitrary finitely generated
  groups. We call a cellular automaton post-surjective if for any pair
  of asymptotic configurations, every pre-image of one is asymptotic
  to a pre-image of the other. The well known dual concept is
  pre-injectivity: a cellular automaton is pre-injective if distinct
  asymptotic configurations have distinct images. We prove that
  pre-injective, post-surjective cellular automata are reversible.
  Moreover, on sofic groups, post-surjectivity alone implies
  reversibility.  We also prove that reversible cellular automata over
  arbitrary groups are balanced, that is, they preserve the uniform
  measure on the configuration space.
\end{abstract}

\section{Introduction}

Cellular automata (briefly, \ca) are parallel synchronous systems on
regular grids where the next state of a point depends on the current
state of a finite neighborhood. The grid is determined by a finitely
generated group and can be visualized as the Cayley graph of the
group.  In addition to being a useful tool for simulations, \ca\ are
studied as models of massively parallel computers, and as dynamical
systems on symbolic spaces.  From a combinatorial point of view,
interesting questions arise as to how the properties of the global
transition function (obtained by synchronous application of the local
update rule at each point) are linked to one another.

One such relation is provided by Bartholdi's theorem~\citep{b10},
which links surjectivity of cellular automata to the preservation of
the product measure on the space of global configurations: the latter
implies the former, but is only implied by it if the grid is an
\emph{amenable} group. In the amenable setting, the \emph{Garden of
  Eden theorem} equates surjectivity with \emph{pre-injectivity}, that
is, the property that two asymptotic configurations (\ie, two
configurations differing on at most finitely many points) with the
same image must be equal. In the same setting, by \citep[Theorem
  4.7]{ff03}, the Garden of Eden theorem still holds for \ca\ on
subshifts that are of finite type and are strongly
irreducible. Counterexamples with general subshifts are known already
in dimension 1. In the general case, the preservation of the product
measure can be expressed combinatorially by the so-called
\emph{balancedness} property. Furthermore, bijectivity is always
equivalent to reversibility, that is, the existence of an inverse that
is itself a \ca.

A parallel to pre-injectivity is \emph{post-surjectivity}, which is
described as follows: given a configuration $e$ and its image $c$,
every configuration $c'$ asymptotic to $c$ has a pre-image $e'$
asymptotic to $e$.  While pre-injectivity is \emph{weaker} than
injectivity, post-surjectivity turns out to be \emph{stronger} than
surjectivity. It is natural to ask whether such trade-off between
injectivity and surjectivity preserves bijectivity.

In this paper, which expands the work presented at the conferences
Automata 2015 and Automata 2016, we discuss the two properties of
balancedness and post-surjectivity, and their links with
reversibility. First, we prove that post-surjectivity and
pre-injectivity together imply reversibility; that is, the trade-off
above actually holds over all groups. Next, we show that, in a context
so broad that no counterexamples are currently known (\ie, when the
grid is a sofic group), post-surjectivity actually implies
pre-injectivity. Finally, we prove that a reversible cellular
automaton over any group is balanced, hence giving an ``almost
positive'' answer to a conjecture proposed in~\citep{cgk13}.

\section{Background}

Given a set $X$, we indicate by $\PF(X)$ the collection of all finite
subsets of $X$. If $X$ is finite, we indicate by $|X|$ the number of
its elements.

Let $\GG$ be a group and let $U,V\subseteq\GG$. We put
\begin{math}
  UV = \{ x \cdot y \mid x \in U, y \in V \} ,
\end{math}
and
\begin{math}
  U^{-1} = \{ x^{-1} \mid x \in U \} .
\end{math}
If $U=\{g\}$ we write $gV$ for $\{g\}V$.

A \emph{labeled graph} is a triple
\begin{math}
  (V, L, E)
\end{math}
where $V$ is a set of \emph{vertices}, $L$ is a set of \emph{labels}, and
\begin{math}
  E \subseteq V \times L \times V
\end{math}
is a set of \emph{labeled edges}.
A \emph{labeled graph isomorphism} from
\begin{math}
  (V_1, L, E_1)
\end{math}
to
\begin{math}
  (V_2, L, E_2)
\end{math}
is a bijection
\begin{math}
  \phi : V_1 \to V_2
\end{math}
such that, for every $x,y\in{V_1}$ and $\ell\in{L}$,
\begin{math}
  (x, \ell, y) \in E_1
\end{math}
if and only if
\begin{math}
  (\phi(x), \ell, \phi(y)) \in E_2 .
\end{math}
We may say that $(V,E)$ is an $L$-labeled graph to mean that $(V,L,E)$
is a labeled graph.

A subset $\gen$ of $\GG$ is a \emph{set of generators} for $\GG$ if
every $g\in\GG$ can be written as
\begin{math}
  g = x_0 \cdots x_{n-1}
\end{math}
for suitable $n\geq0$ and
\begin{math}
  x_0, \ldots, x_{n-1} \in \gen \cup \gen^{-1} .
\end{math}
The group $\GG$ is \emph{finitely generated} (briefly, f.g.) if $\gen$
can be chosen finite.

Let $\gen$ be a finite set of generators for the group $\GG$. The
\emph{Cayley graph} of~$\GG$ with respect to $\gen$ is the labeled graph
$(\GG,L,E)$ where
\begin{math}
  L = (\gen \cup \gen^{-1})
\end{math}
and
\begin{math}
  E = \{ (g, x, h) \mid gx = h \} .
\end{math}
The \emph{length} of $g\in\GG$ \wrt $\gen$ is the \emph{minimum}
length
\begin{math}
  n = \|g\|_\gen
\end{math}
of a representation
\begin{math}
  g = x_0 \cdots x_{n-1} .
\end{math}
The \emph{distance} between $g$ and $h$ with respect to $\gen$ is
\begin{math}
  d_\gen(g, h) = \|g^{-1} \cdot h\|_\gen ,
\end{math} 
\ie, the length of the shortest path from $g$ to $h$ in the Cayley
graph of $\GG$ \wrt $\gen$. With respect to such distance,
multiplications to the left by a fixed element, \ie, the functions
$x\mapsto{gx}$ where $g\in\GG$ is fixed, are isometries.
The \emph{disk} of center $g$ and radius $r$ \wrt $\gen$ is the set
$D_{\gen,r}(g)$ of those $h\in\GG$ such that
\begin{math}
 d_\gen(g, h) \leq r .
\end{math}
We omit $g$ if it is the identity element $1_\GG$ of $\GG$ and write
$D_{\gen, r}$ for $D_{\gen,r}(1_\GG)$.
The distance between two subsets $U,V$ of $\GG$ with respect to $\gen$ is
\begin{math}
  d_\gen(U, V) = \inf \{ d_\gen(u, v) \mid u \in U, v \in V \} .
\end{math}
We omit $\gen$ if irrelevant or clear from the context.

A group $\GG$ is \emph{amenable} if for every
\begin{math}
  K \in \PF(\GG)
\end{math}
and every $\varepsilon>0$ there exists
\begin{math}
  F \in \PF(\GG)
\end{math}
such that
\begin{math}
  |F \cap kF| > (1 - \varepsilon) |F|
\end{math}
for every $k\in{K}$. The groups $\Zset^d$ are amenable, whereas the
\emph{free groups} on two or more generators are not. For an
introduction to amenability see, \eg, \citep[Chapter 4]{csc10}.

Let $S$ be a finite set and let $\GG$ be a group. The elements of the
set $S^\GG$ are called \emph{configurations}. The space $S^\GG$ is
given the \emph{prodiscrete topology} by considering $S$ as a discrete
set. This makes $S^\GG$ a compact space by Tychonoff's theorem. In the
prodiscrete topology, two configurations are ``near'' if they coincide
on a ``large'' finite subset of $\GG$. If $\GG$ is countable, then the
prodiscrete topology is metrizable: indeed, if
\begin{math}
  \GG = \{ g_n \}_{n \geq 0}
\end{math},
then  
\begin{math}
  d(c, e) = 2^{-n}
\end{math}
for all $c,e\in{S^\GG}$, where $n\geq0$ is the smallest index such that
\begin{math}
  c(g_n) \neq e(g_n)
\end{math},
is a distance that induces the product topology. If, in addition,
$\gen$ is a finite set of generators for $\GG$, then setting
\begin{math}
  d_\gen(c, e) = 2^{-n} ,
\end{math}
for all $c,e\in{S^\GG}$, where $n$ is the smallest non-negative
integer such that $c$ and $e$ differ on $D_{\gen,n}$, also defines a
distance that induces the prodiscrete topology. Given
\begin{math}
  c, c' \in S^\GG ,
\end{math}
we call
\begin{math}
  \Delta(c, c') = \{ g \in \GG \mid c(g) \neq c'(g) \}
\end{math}
the \emph{difference set} of $c$ and $c'$. Two configurations are
\emph{asymptotic} if they differ at most on finitely many points of
$\GG$.  A \emph{pattern} is a function
\begin{math}
  p : E \to S
\end{math}
where $E$ is a finite subset of $\GG$.

For $g\in\GG$, the \emph{translation} by $g$ is the function
\begin{math}
  \sigma_g : S^\GG \to S^\GG
\end{math}
that sends an arbitrary configuration $c$ into the configuration
$\sigma_g(c)$ defined by
\begin{equation}
  \label{eq:transl}
  \sigma_g(c)(x) = c(g \cdot x)
  \;\; \forall x \in \GG
  \,.
\end{equation}
A \emph{shift subspace} (briefly, \emph{subshift}) is a subset $X$ of
$S^\GG$ which is closed (equivalently, compact) and invariant by all
translations $\sigma_g$ with $g\in\GG$.  The set $S^\GG$ itself is referred to as the
\emph{full shift}.
It is well known (see e.g.~\citep{lm95}) that every subshift $X$ is
determined by a set of \emph{forbidden patterns} $\Forb$ in the sense
that the elements of the subshift $X$ are precisely those
configurations in which translations of patterns in $\Forb$ do not
occur.  If $\Forb$ can be chosen finite, $X$ is called a \emph{shift
  of finite type} (briefly, \sft).
A pattern $p:E\to{S}$ is said to be \emph{admissible} for $X$ if there
exists a configuration $c\in{X}$ such that
\begin{math}
  \restrict{c}{E} = p
\end{math}.
The set of patterns that are admissible for $X$ is called the
\emph{language} of $X$, indicated as $\Lang_X$.

A \emph{cellular automaton} (briefly, \ca) on a group $\GG$ is a
triple
\begin{math}
  \Acal = \langle S, \Neigh, f \rangle
\end{math}
where the \emph{set of states} $S$ is finite and has at least two
elements, the \emph{neighborhood} $\Neigh$ is a finite subset of
$\GG$, and the \emph{local update rule} is a function that associates
to every pattern
\begin{math}
  p : \Neigh \to S
\end{math}
a state $f(p)\in S$. The \emph{global transition function} of $\Acal$
is the function
\begin{math}
  F_\Acal : S^\GG \to S^\GG
\end{math}
defined by
\begin{equation} \label{eq:gtf}
  F_\Acal(c) (g)
  =
  f \left( \restrict{(\sigma_g(c))}{\Neigh} \right)
  \;\;
  \forall g \in \GG
  \,:
\end{equation}
that is, if
\begin{math}
  \Neigh = \{ n_1, \ldots, n_m \} ,
\end{math}
then
\begin{math}
  F_\Acal(c)(g) = f \left( c(g \cdot n_1), \ldots, c(g \cdot n_m) \right) .
\end{math}
Observe that (\ref{eq:gtf}) is continuous in the prodiscrete topology
and commutes with the translations, \ie,
\begin{math}
  F_\Acal \circ \sigma_g = \sigma_g \circ F_\Acal
\end{math}
for every $g \in \GG$. The \emph{Curtis-Hedlund-Lyndon theorem} states
that the continuous and translation-commuting functions from $S^\GG$
to itself are precisely the \ca\ global transition functions.

We shall use the following notation to represent the application of
the local rule on patterns.  If
\begin{math}
  p : E \to S
\end{math}
and
\begin{math}
  q : C \to S
\end{math}
are two patterns, we write
\begin{math}
  p \update{f} q
\end{math}
to indicate that
\begin{math}
  C\Neigh \subseteq E
\end{math}
and
\begin{math}
  q(g) = f \left( \restrict{(\sigma_g(p))}{\Neigh} \right)
\end{math}
for each $g\in C$.

If $X$ is a subshift and $F_\Acal$ is a cellular automaton, it is easy
to see that $F_\Acal(X)$ is also a subshift. If, in addition,
\begin{math}
  F_\Acal(X) \subseteq X ,
\end{math}
we say that $\Acal$ is a \ca\ on the subshift $X$. From now on, when
we speak of cellular automata on $\GG$ without specifying any
subshift, we will imply that such subshift is the full shift.

We may refer to injectivity, surjectivity, etc. of the cellular
automaton $\Acal$ on the subshift $X$ meaning the corresponding
properties of $F_\Acal$ when restricted to $X$.  Since $X$ is a
compact metric space, it follows from the Curtis-Hedlund-Lyndon
theorem that the inverse of the global transition function of a
bijective cellular automaton $\Acal$ is itself the global transition
function of some cellular automaton.  In this case, we say that
$\Acal$ is \emph{reversible}. A group $\GG$ is \emph{surjunctive} if
for every finite set $S$, every injective cellular automaton on the
full shift $S^\GG$ is surjective. Currently, there are no known
examples of non-surjunctive groups.

\begin{conjecture}[{\citealp{gott73}}]
  Every injective \ca\ on a full shift is surjective.
\end{conjecture}

If $\GG$ is a subgroup of a group $\Gamma$ and
\begin{math}
  \Acal = \langle S, \Neigh, f \rangle
\end{math}
is a cellular automaton on $\GG$, the cellular automaton
$\Acal^\Gamma$ \emph{induced} by $\Acal$ on $\Gamma$ has the same set
of states, neighborhood, and local update rule as $\Acal$, and maps
$S^\Gamma$ (instead of $S^\GG$) into itself via
\begin{math}
  F_{\Acal^\Gamma}(c)(\gamma)
  =
  f \left( c(\gamma \cdot n_1), \ldots, c(\gamma \cdot n_m) \right)
\end{math}
for every $\gamma \in \Gamma$. We also say that $\Acal$ is the
\emph{restriction} of $\Acal^\Gamma$ to $\GG$.
In addition, if $X\subseteq{S^\GG}$ is a subshift defined by a set
$\Forb$ of forbidden patterns on $\GG$, then the subshift
$X^\Gamma\subseteq{S^\Gamma}$ obtained from the same set $\Forb$ of
forbidden patterns satisfies the following property: if $\Acal$ is a
\ca\ on $X$, then $\Acal^\Gamma$ is a \ca\ on $X^\Gamma$, and vice
versa. (Here, it is fundamental that all the forbidden patterns have
their supports in $\GG$.)
It turns out (see~\citep[Lemma 4.3]{c09}) that induction of subshifts
does not depend on the choice of $\Forb$, and that injectivity and
surjectivity are preserved by both induction and restriction (see
also~\citep[Section 1.7]{csc10} and~\citep[Theorem 5.3]{c09}).

Let
\begin{math}
  \Acal = \langle S, \Neigh, f \rangle
\end{math}
be a \ca\ on a subshift $X$, let
\begin{math}
  p : E \to S
\end{math}
be an admissible
pattern for $X$, and let
\begin{math}
  E \Neigh \subseteq M \in \PF(\GG) .
\end{math}
A \emph{pre-image} of $p$ on $M$ under $\Acal$ is a pattern
\begin{math}
  q : M \to S
\end{math}
that is admissible for $X$ and is such that
\begin{math}
q \update{f} p .
\end{math}
An \emph{orphan} is an admissible pattern that has no admissible
pre-image, or equivalently, a pattern that is admissible for $X$ but
not admissible for $F_\Acal(X)$. Similarly, a configuration which is
not in the image of $X$ by $F_\Acal$ is a \emph{Garden of Eden} for
$\Acal$. By a compactness argument, every Garden of Eden contains an
orphan. We call this fact the \emph{orphan pattern principle}.
A cellular automaton $\Acal$ is \emph{pre-injective} if every two
asymptotic configurations $c, e$ satisfying $F_\Acal(c)=F_\Acal(e)$
are equal. The \emph{Garden of Eden theorem} (see~\citep{csms99})
states that, for \ca\ on amenable groups, pre-injectivity is
equivalent to surjectivity; on non-amenable groups, the two properties
are independent of each other (see~\citep{b10} and~\citep{b16}).

Let $\GG$ be a finitely generated group, let $\gen$ be a finite set of
generators for $\GG$, and let $S$ be a finite set.  A subshift
\begin{math}
  X\subseteq S^\GG
\end{math}
is \emph{strongly irreducible} if there exists $r\geq1$ such that,
for every two admissible patterns
\begin{math}
  p_1 : E_1 \to S, p_2 : E_2 \to S
\end{math}
such that
\begin{math}
  d_\gen(E_1, E_2) \geq r ,
\end{math}
there exists $c \in X$ such that
\begin{math}
  \restrict{c}{E_1} = p_1
\end{math}
and
\begin{math}
  \restrict{c}{E_2} = p_2 .
\end{math}
We then say that $r$ is a \emph{constant of strong irreducibility} for
$X$ with respect to~$\gen$.  The notion of strong irreducibility does
not depend on the choice of the finite set of generators, albeit the
associated constant of strong irreducibility usually does. If no
ambiguity is possible, we will suppose $\gen$ fixed once and for all,
and always speak of $r$ relative to $\gen$.  For $\GG=\Zset$, strong
irreducibility is equivalent to existence of $r\geq1$ such that, for
every two
\begin{math}
  u, v \in \Lang_X ,
\end{math}
there exists $w\in{S^r}$ satisfying
\begin{math}
  u w v \in \Lang_X .
\end{math}
Clearly, every full shift is strongly irreducible.

As a consequence of the definition, strongly irreducible subshifts are
\emph{mixing}: given two open sets $U,V\subseteq{X}$, the set of those
$g\in\GG$ such that
\begin{math}
  U \cap \sigma_g^{-1}(V) = \emptyset
\end{math}
is, at most, finite.
In addition to this, as by~\citep[Theorem 8.1.16]{lm95}, the Garden of
Eden theorem is still valid on strongly irreducible subshifts of
finite type. We remark that for one-dimensional subshifts of finite
type, strong irreducibility is equivalent to the mixing property.

Another property of strongly irreducible subshifts, which will have a
crucial role in the next section, is that they allow a ``cut and
paste'' technique which is very common in proofs involving the full
shift, but may be inapplicable for more general shifts.

\begin{proposition}
  \label{prop:cut-and-paste}
  Let
  \begin{math}
    X \subseteq S^\GG
  \end{math}
  be a strongly irreducible subshift, let $c\in{X}$, and let
  \begin{math}
    p : E \to S
  \end{math}
  be an admissible pattern for $X$. There exists $c'\in{X}$ asymptotic
  to $c$ such that
  \begin{math}
    \restrict{c'}{E} = p .
  \end{math}
\end{proposition}

\begin{proof}
  It is not restrictive to suppose $E=D_n$ for suitable $n\geq0$.  Let
  $r\geq1$ be a constant of strong irreducibility for $X$.  Writing
  \begin{math}
    E_k = D_{n+r+k} \setminus D_{n+r}
  \end{math}
  for $k\geq1$, we have of course
  \begin{math}
    d(E, E_k) = r .
  \end{math}
  Set
  \begin{math}
    p_k = \restrict{c}{E_k} .
  \end{math}
  By strong irreducibility, there exists $c_k\in{X}$ such that
  \begin{math}
    \restrict{c_k}{E} = p
  \end{math}
  and
  \begin{math}
    \restrict{c_k}{E_k} = p_k .
  \end{math}
  Then every limit point $c'$ of
  \begin{math}
    \{c_k\}_{k \geq 1} ,
  \end{math}
  which exists and belongs to $X$ because of compactness, satisfies
  the thesis.
\end{proof}

Induction and restriction do not affect strong irreducibility.

\begin{proposition}
  Let $\GG$ and $\Gamma$ be finitely generated groups, where $\GG$ is
  a subgroup of $\Gamma$, and let $S$ be a finite set. Let
  \begin{math}
    X \subseteq S^\GG
  \end{math}
  be a subshift and let
  \begin{math}
    X^\Gamma\subseteq S^\Gamma
  \end{math}
  be the subshift induced by $X$. If one between $X$ and $X^\Gamma$ is
  strongly irreducible, so is the other.
\end{proposition}

\begin{proof}
  To fix ideas, let $\gen_\GG$ and $\gen_\Gamma$ be two finite sets of
  generators for $\GG$ and $\Gamma$, respectively, let $J$ be a set of
  representatives of the left cosets of $\GG$ in $\Gamma$, so that
  \begin{math}
    \Gamma = \bigsqcup_{j \in J} j\GG 
  \end{math},
  and let $\Forb$ be a set of forbidden patterns that determines $X$.

  Suppose that $X^\Gamma$ is strongly irreducible and $r\geq1$ a
  constant of strong irreducibility for $X^\Gamma$. Take $r'\geq1$
  such that
  \begin{math}
    D_{B_\Gamma, r-1} \cap \GG \subseteq D_{B_\GG, r'-1}
  \end{math},
  which exists because the left-hand side is finite. Let
  \begin{math}
    E_1, E_2 \subseteq \GG
  \end{math}
  satisfy
  \begin{math}
    d_{B_\GG}(E_1, E_2) \geq r'
  \end{math}.
  Then, by construction,
  \begin{math}
    d_{B_\Gamma}(E_1, E_2) \geq r
  \end{math}
  too. Given two admissible patterns
  \begin{math}
    p_1 : E_1 \to S
  \end{math},
  \begin{math}
    p_2 : E_2 \to S
  \end{math},
  take $c\in{X^\Gamma}$ such that
  \begin{math}
    \restrict{c}{E_1} = p_1
  \end{math}
  and
  \begin{math}
    \restrict{c}{E_2} = p_2
  \end{math}.
  Then
  \begin{math}
    \restrict{c}{\GG} \in X
  \end{math}
  has the same property.

  Next, suppose that $X$ is strongly irreducible and $r\geq1$ is a
  constant of strong irreducibility for $X$. Let $M\geq 1$ be such
  that every element of $\gen_\GG$ can be written as a product of at
  most $M$ elements of $\gen_\Gamma$. Then $Mr$ is a constant of
  strong irreducibility for $X^\Gamma$. Indeed, let
  \begin{math}
    p_1 : E_1 \to S, p_2 : E_2 \to S
  \end{math}
  be two admissible patterns such that
  \begin{math}
    d_{\gen_\Gamma}(E_1, E_2) \geq Mr
  \end{math}.
  For $i=1,2$, there exist at most finitely many $j\in{J}$ such that
  \begin{math}
    E_{i,j} = E_i \cap j\GG \neq \emptyset .
  \end{math}
  If for a given $j$ both $E_{1,j}$ and $E_{2,j}$ are nonempty, then
  \begin{math}
    d_{B_\Gamma}(E_{1,j}, E_{2,j}) \geq d_{B_\Gamma}(E_1, E_2) \geq Mr ,
  \end{math}
  hence, since $\gen_\GG\subseteq\gen_\Gamma$ and multiplications on
  the left are isometries,
  \begin{math}
    d_{B_\GG}(j^{-1} E_{1,j}, j^{-1} E_{2,j}) \geq r
  \end{math}
  by definition of $M$. We can then construct a configuration
  $c\in{S^\Gamma}$ such that
  \begin{math}
    \restrict{c}{E_1} = p_1
  \end{math}
  and
  \begin{math}
    \restrict{c}{E_2} = p_2
  \end{math}
  as follows:
  \begin{itemize}
  \item
    If $x\in{j\GG}$ and both $E_{1,j}$ and $E_{2,j}$ are nonempty, let
    \begin{math}
      c(x) = c_j(j^{-1} x) ,
    \end{math}
    where $c_j\in{X}$ is such that
    \begin{math}
      c_j(j^{-1}x) = p_1(x)
    \end{math}
    if $x\in{E_{1,j}}$ and
    \begin{math}
      c_j(j^{-1}x) = p_2(x)
    \end{math}
    if $x\in{E_{2,j}}$.
  \item
    If $x\in{j\GG}$ and, of $E_{1,j}$ and $E_{2,j}$, one is nonempty
    and the other is empty, then, calling $E$ the nonempty one and $p$
    the corresponding pattern, let
    \begin{math}
      c(x) = c_j(j^{-1} x) ,
    \end{math}
    where $c_j\in{X}$ is such that
    \begin{math}
      c_j(j^{-1}x) = p(x)
    \end{math}
    for every $x\in{E}$.
  \item
    If $x\in{j\GG}$ and $E_{1,j}$ and $E_{2,j}$ are both empty, let
    \begin{math}
      c(x) = \bar{c}(j^{-1} x)
    \end{math}
    where $\bar{c}\in{X}$ is fixed.
  \end{itemize}
  It is easy to see that no pattern from $\Forb$ can have any
  occurrences in $c$, so that $c\in X^\Gamma$.
\end{proof}

\section{Post-surjectivity}

The notion of post-surjectivity is a sort of ``dual'' to
pre-injectivity: it is a strengthening of surjectivity, in a similar
way that pre-injectivity is a weakening of injectivity. The maps that
are both pre-injective and post-surjective were studied
in~\citep{kt15} under the name of complete pre-injective maps.

\begin{definition}
  \label{def:postsurj}
  Let $\GG$ be a group, $S$ a finite set, and
  \begin{math}
    X \subseteq S^\GG
  \end{math}
  a strongly irreducible subshift. A cellular automaton
  \begin{math}
    \Acal = \langle S, \Neigh, f \rangle
  \end{math}
  on $X$ is \emph{post-surjective} if, however given $c\in{X}$ and a
  predecessor $e\in{X}$ of $c$, every configuration $c'\in{X}$
  asymptotic to $c$ has a predecessor $e'\in{X}$ asymptotic to $e$.
\end{definition}

When $X = S^\GG$ is the full shift, if no ambiguity is present, we
will simply say that the \ca\ is post-surjective.

\begin{example}
  \label{ex:revers-postsurj}
  Every reversible cellular automaton is post-surjective.
  If $R\geq0$ is such that the neighborhood of the inverse \ca\ is
  included in $D_R$, and $N\geq0$ is such that $c$ and $c'$ coincide
  outside $D_N$, then their unique pre-images $e$ and $e'$ must
  coincide outside $D_{N+R}$.
  \hfill\exampleqed
\end{example}

\begin{example}
  \label{ex:xor}
  The $\xor$ \ca\ with the right-hand neighbor (the one-dimensional
  elementary \ca\ with rule 102) is surjective, but not
  post-surjective.
  As the $\xor$ function is a permutation of each of its arguments
  given the other, every
  \begin{math}
    c \in \deux^\Zset
  \end{math}
  has two pre-images, uniquely determined by their value in a single
  point.  However (actually: because of this!)
  \begin{math}
    \ldots 000 \ldots
  \end{math}
  is a fixed point, but
  \begin{math}
    \ldots 010 \ldots
  \end{math}
  only has pre-images that take value 1 infinitely often.
  \hfill\exampleqed
\end{example}

The qualification ``post-surjective'' is well earned:

\begin{proposition}
  \label{prop:post-surj}
  Let
  \begin{math}
    X \subseteq S^\GG
  \end{math}
  be a strongly irreducible subshift. Every post-surjective \ca\ on
  $X$ is surjective.
\end{proposition}
\begin{proof}
  Let $r\geq1$ be the constant of strong irreducibility of $X$, \ie,
  let every two admissible patterns whose supports have distance
  at least $r$ be jointly subpatterns of some configuration.
  Take an arbitrary $e\in{X}$ and set $c=F(e)$. Let $p:E\to{S}$ be an
  admissible pattern for $X$. By Proposition \ref{prop:cut-and-paste},
  there exists $c'\in{X}$ asymptotic to $c$ such that
  \begin{math}
    \restrict{c'}{E} = p .
  \end{math}
  By post-surjectivity, such $c'$ has a pre-image in $X$, which means
  $p$ has a pre-image admissible for $X$. The thesis follows from the
  orphan pattern principle.
\end{proof}

From Proposition~\ref{prop:post-surj} together with \citep[Theorem
  4.7]{ff03} follows:
\begin{proposition}
  \label{prop:postsurj-amen}
  Let $\GG$ be a finitely generated amenable group and let
  \begin{math}
    X \subseteq S^\GG
  \end{math}
  be a strongly irreducible \sft. Every post-surjective \ca\ on $X$ is
  pre-injective.
\end{proposition}

In addition, via a reasoning similar to the one employed
in~\citep[Section 1.7]{csc10} and~\citep[Remark 18]{cgk13}, we can
prove:

\begin{proposition}
  \label{prop:postsurj-restrict}
  Let $\GG$ and $\Gamma$ be finitely generated groups where $\GG$ is a
  subgroup of $\Gamma$.  Let
  \begin{math}
    X \subseteq S^\GG
  \end{math}
  be a strongly irreducible subshift and let
  \begin{math}
    X^\Gamma \subseteq S^\Gamma
  \end{math}
  be the induced subshift. Let
  \begin{math}
    \Acal = \langle S, \Neigh, f \rangle
  \end{math}
  be a cellular automaton on $X$ and $\Acal^\Gamma$ the induced
  cellular automaton on $X^\Gamma$. Then $\Acal$ is post-surjective if
  and only if $\Acal^\Gamma$ is post-surjective.
\end{proposition}
In particular, post-surjectivity of arbitrary \ca\ is equivalent to
post-surjectivity on the subgroup generated by the neighborhood.

\begin{proof}
  Suppose $\Acal$ is post-surjective. Let $J$ be a set of
  representatives of the left cosets of $\GG$ in $\Gamma$, \ie, let
  \begin{math}
    \Gamma = \bigsqcup_{j \in J} j\GG .
  \end{math}
  Let
  \begin{math}
    c, c' \in X^\Gamma
  \end{math}
  be two asymptotic configurations and let $e$ be a pre-image of
  $c$. For every $j\in J$ and $g\in\GG$ set
  \begin{eqnarray*}
    c_j(g) & = & c(jg) \,;
    \\
    c'_j(g) & = & c'(jg) \,;
    \\
    e_j(g) & = & e(jg) \,.
  \end{eqnarray*}
  By construction, each $c_j$ belongs to $X$, is asymptotic to $c'_j$
  and has $e_j$, which also belongs to $X$, as a pre-image according
  to $\Acal$.
  Moreover, as $c$ and $c'$ are asymptotic in the first place,
  $c'_j\neq c_j$ only for finitely many $j\in{J}$. For every $j\in{J}$
  let $e'_j\in{X}$ be a pre-image of $c_j'$ according to $\Acal$
  asymptotic to $e_j$, if $c'_j\neq c_j$, and $e_j$ itself if
  $c'_j=c_j$.  Then,
  \begin{displaymath}
    e'(\gamma) = e'_j(g)
    \;\; \iff \;\; \gamma = jg
  \end{displaymath}
  defines a pre-image of $c'$ according to $\Acal^\Gamma$ which
  belongs to $X^\Gamma$ and is asymptotic to $e$.

  The converse implication is immediate.
\end{proof}

\begin{proposition}
  \label{prop:postsurj-1d}
  Let $X\subseteq{S^\Zset}$ be a strongly irreducible \sft\ and let
  \begin{math}
    \Acal = \langle S, \Neigh, f \rangle
  \end{math}
  be a post-surjective \ca\ on $X$. Then $\Acal$ is reversible.
\end{proposition}
\begin{proof}
  Suppose $F=F_\Acal$ is not a bijection. For \ca\ on one-dimensional
  strongly irreducible \sft, reversibility is equivalent to
  injectivity on periodic configurations. Namely, if two distinct
  configurations with the same image exist, then one can construct two
  distinct \emph{periodic} configurations with the same image. Let
  then
  \begin{math}
    u, v, w \in S^\ast
  \end{math}
  be such that
  \begin{math}
    e_u = \ldots u u u \ldots ,
  \end{math}
  the configuration obtained by extending $u$ periodically in both
  directions, and
  \begin{math}
    e_v = \ldots v v v \ldots
  \end{math}
  are different and have the same image
  \begin{math}
    c = \ldots w w w \ldots .
  \end{math}
  It is not restrictive to suppose $|u|=|v|=|w|$. Without loss of
  generality, we also assume that $X$ is defined by a set of forbidden
  words of length at most $|u|$.

  Let $r\geq1$ be a strong irreducibility constant for $X$ and let
  \begin{math}
    p, q \in S^r
  \end{math}
  be such that
  \begin{math}
    u p v , v q u \in \Lang_X .
  \end{math}
  The two configurations
  \begin{math}
    c_{u,v} = F( \ldots u u p v v \ldots )
  \end{math}
  and
  \begin{math}
    c_{v,u} = F( \ldots v v q u u \ldots )
  \end{math}
  are both asymptotic to $c$. By post-surjectivity, there exist
  $x,y\in\Lang_X$ such that
  \begin{math}
    e_{u,v} = \ldots u u x v v \ldots
  \end{math}
  and
  \begin{math}
    e_{v,u} = \ldots v v y u u \ldots
  \end{math}
  satisfy
  \begin{math}
    F( e_{u,v} ) = F( e_{v,u} ) = c .
  \end{math}
  Again, it is not restrictive to suppose that
  \begin{math}
    |x| = |y| = m \cdot |u|
  \end{math}
  for some $m \geq 1$, and that $x$ and $y$ start in $e_{u,v}$ and
  $e_{v,u}$ at the same point $i\in\Zset$.

  Let us now consider the configuration
  \begin{math}
    e' = \ldots u u x v^N y u u \ldots .
  \end{math}
  By our previous discussion, for $N$ large enough (\eg, so that $x$
  and $y$ do not have overlapping neighborhoods) $F_\Acal(e')$ cannot
  help but be $c$.  Now, recall that $e_u$ is also a pre-image of $c$
  and note that $e_u$ and $e'$ are asymptotic but distinct. Then
  $\Acal$ is surjective (by Proposition~\ref{prop:post-surj}) but not
  pre-injective, contradicting the Garden of Eden
  theorem~\citep[Theorem 8.1.16]{lm95} as well as
  Proposition~\ref{prop:postsurj-amen}.

  A graphical description of the argument is provided by
  Figure~\ref{fig:postsurj-1d}.
\end{proof}

\begin{figure}
  \centering
  
  \begin{subfigure}[b]{0.45\textwidth}
  	\centering
    \includegraphics[width=0.9\textwidth]{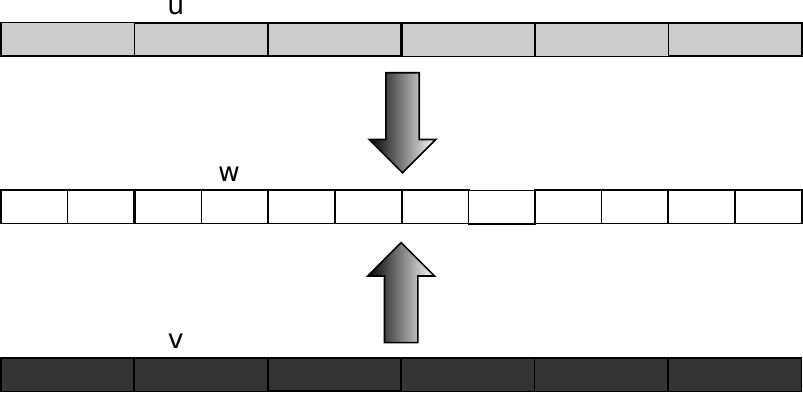}
    \caption{%
    }
    \label{subfig:1d-ps-rev-nonrev}
  \end{subfigure}
  \begin{subfigure}[b]{0.45\textwidth}
  	\centering
    \includegraphics[width=0.9\textwidth]{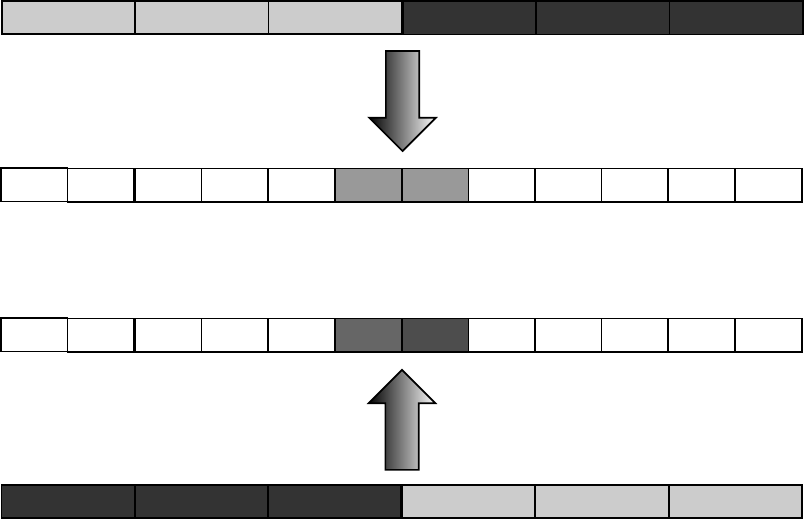}
    \caption{%
    }
    \label{subfig:1d-ps-rev-swaphalves}
  \end{subfigure}

  \bigskip

  \begin{subfigure}[b]{0.45\textwidth}
  	\centering
    \includegraphics[width=0.9\textwidth]{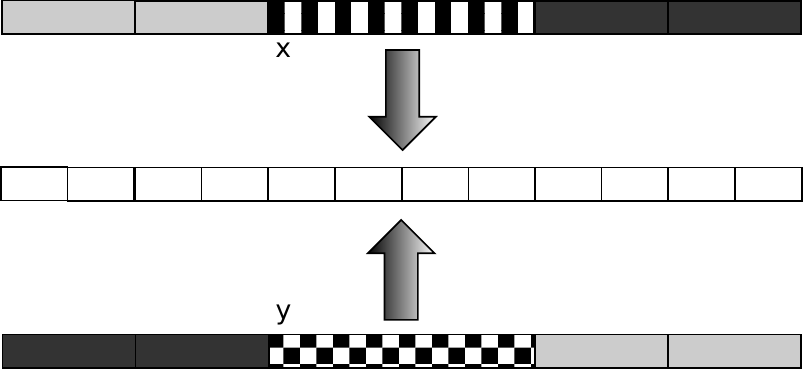}
    \caption{%
    }
    \label{subfig:1d-ps-rev-postsurj}
  \end{subfigure}
  \begin{subfigure}[b]{0.45\textwidth}
  	\centering
    \includegraphics[width=0.9\textwidth]{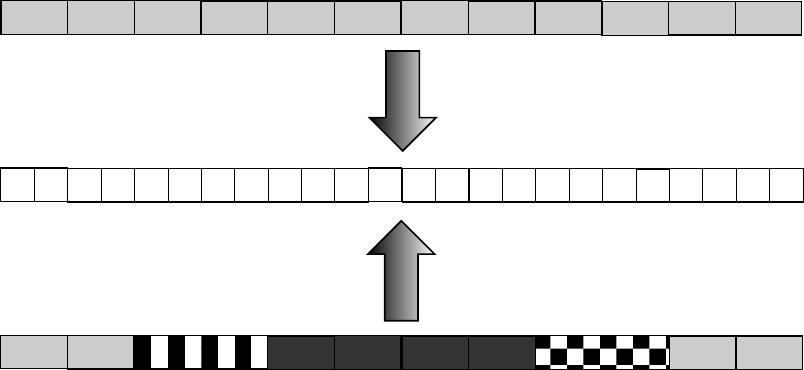}
    \caption{%
    }
    \label{subfig:1d-ps-rev-goe}
  \end{subfigure}
  \caption{%
    A graphical description of the argument in
    Proposition~\ref{prop:postsurj-1d} for the full shift.
    (\subref{subfig:1d-ps-rev-nonrev})
    Let a 1D periodic configuration $w$ have two different (periodic)
    preimages $u$ and $v$.
    (\subref{subfig:1d-ps-rev-swaphalves})
    By swapping the right-hand halves of the preimages, the new images
    only differ from the initial one in finitely many points.
    (\subref{subfig:1d-ps-rev-postsurj})
    By post-surjectivity, we can change them in finitely many points,
    and get two preimages of the initial configuration.
    (\subref{subfig:1d-ps-rev-goe})
    Then a violation of the Garden of Eden theorem occurs.
  }
  \label{fig:postsurj-1d}
\end{figure}

Proposition~\ref{prop:postsurj-1d} depends critically on the group
being $\Zset$, where \ca\ that are injective on periodic
configurations are reversible. Moreover, in our final step, we invoke
the Garden of Eden theorem, which we know from~\citep{csms99} not to
hold for \ca\ on generic groups. Not all is lost, however: maybe, by
explicitly adding the pre-injectivity requirement, we can recover
Proposition~\ref{prop:postsurj-1d} on more general groups.

It turns out that it is so, at least for \ca\ on full shifts. To see
this, we need some preparations.

\begin{lemma}
  \label{lem:postsurj-N}
  Let $\Acal$ be a post-surjective \ca\ on a finitely generated group
  $\GG$ and let $F$ be its global transition function.  There exists
  $N\geq0$ such that, given any three configurations
  \begin{math}
    c, c', e
  \end{math}
  with $c=F(e)$ and
  \begin{math}
    \Delta(c, c') = \{ 1_\GG \} ,
  \end{math}
  there exists a pre-image $e'$ of $c'$ which coincides with $e$
  outside $D_N$.
\end{lemma}
\begin{proof}
  By contradiction, assume that for every $n\geq0$, there exist
  \begin{math}
    c_n, c'_n \in S^\GG 
  \end{math}
  and
  \begin{math}
    e_n \in F^{-1}(c_n)
  \end{math}
  such that
  \begin{math}
    \Delta(c_n, c'_n) = \{ 1_\GG \} ,
  \end{math}
  but every
  \begin{math}
    e'_n \in F^{-1}(c'_n)
  \end{math}
  differs from $e_n$ on some point outside $D_n$.  By compactness,
  there exits a sequence $n_i$ such that the limits
  \begin{math}
    c = \lim_{i \to \infty} c_{n_i} ,
    c' = \lim_{i \to \infty} c'_{n_i} ,
  \end{math}
  and
  \begin{math}
    e = \lim_{i \to \infty} e_{n_i} ,
  \end{math}
  all exist. Then $F(e)=c$ by continuity. By construction, $c$
  differs from $c'$ only at $1_\GG$. By post-surjectivity, there
  exists a pre-image $e'$ of $c'$ such that
  \begin{math}
  \Delta(e,e') \subseteq D_m
  \end{math}
  for some $m\geq0$.  Take $\ell\gg{m}$ and choose $k$ large enough
  such that
  \begin{math}
    \restrict{c'_{n_k}}{D_{\ell}} = \restrict{c'}{D_{\ell}}
  \end{math}
  and
  \begin{math}
    \restrict{e_{n_k}}{D_{\ell}} = \restrict{e}{D_{\ell}} .
  \end{math}
  Define $\tilde{e}$ so that it agrees with $e'$ on $D_{\ell}$ and
  with $e_{n_k}$ outside $D_m$. Such $\tilde{e}$ is well defined,
  because $e'$, $e$, and $e_{n_k}$ agree on
  \begin{math}
    D_{\ell} \setminus D_m .
  \end{math}
  Then $\tilde{e}$ is a pre-image of $c'_{n_k}$ which is asymptotic to
  $e_{n_k}$ and agrees with $e_{n_k}$ outside $D_{n_k}$, thus
  contradicting our assumption.
\end{proof}

By repeatedly applying Lemma~\ref{lem:postsurj-N} we get:
\begin{proposition}
  \label{prop:postsurj-N}
  Let $\Acal$ be a post-surjective \ca\ on a finitely generated group
  $\GG$ and let $F$ be its global transition function. There exists
  $N\geq 0$ such that, for every $r \geq 0$, however given three
  configurations $c,c',e$ with $c=F(e)$ and
  \begin{math}
    \Delta(c,c') \subseteq D_r ,
  \end{math}
  there exists a pre-image $e'$ of $c'$ such that
  \begin{math}
    \Delta(e, e') \subseteq D_{N+r} .
  \end{math}
\end{proposition}

Assuming also pre-injectivity, we get the following stronger property:
\begin{corollary}
  \label{cor:tk}
  Let $\Acal$ be a pre-injective, post-surjective \ca\ on a finitely
  generated group $\GG$ and let $F$ be its global transition function.
  There exists $M \in \PF(\GG)$ with the following property: for every
  pair $(e,e')$ of asymptotic configurations,
  \begin{math}
    \Delta(e, e') \subseteq \Delta(F(e), F(e)') M
  \end{math}.
\end{corollary}

We are now ready to prove:
\begin{theorem}
  \label{thm:pre-post-rev}
  Every pre-injective, post-surjective cellular automaton on a full
  shift is reversible.
\end{theorem}
\begin{proof}
  By Proposition~\ref{prop:postsurj-restrict}, it is sufficient to
  consider the case where $\GG$ is finitely generated.

  Let $\Acal$ be a pre-injective and post-surjective \ca\ on the group
  $\GG$, let $S$ be its set of states, and let $F$ be its global
  transition function. Let $M$ be as in Corollary~\ref{cor:tk}.
  We construct a new \ca\ with neighborhood $\Neigh=M^{-1}$. Calling
  $H$ the global transition function of the new \ca, we first prove
  that $H$ is a \emph{right} inverse of $F$. We then show that $H$ is
  also a \emph{left} inverse for $F$, thus completing the proof.

  To construct the local update rule
  \begin{math}
    h : S^\Neigh \to S ,
  \end{math}
  we proceed as follows. Fix a constant configuration $u$ and let
  $v=F(u)$. Given $g\in\GG$ and
  \begin{math}
    p : \Neigh \to S ,
  \end{math}
  for every $i\in\GG$, put
  \begin{equation} \label{eq:cut-and-paste}
    y_{g,p}(i) =
    \ifthenelsedef{i \in g \Neigh}{p(g^{-1} i)}{v(i)}
  \end{equation}
  that is, let $y_{g,p}$ be obtained from $v$ by cutting away the
  piece with support $g \Neigh$ and pasting $p$ as a ``patch'' for the
  ``hole''. By post-surjectivity and pre-injectivity combined, there
  exists a unique
  \begin{math}
    x_{g,p} \in S^\GG
  \end{math}
  asymptotic to $u$ such that
  \begin{math}
  F(x_{g,p}) = y_{g,p} .
  \end{math}
  Let then
  \begin{equation} \label{eq:pi-ps-inverse}
    h(p) = x_{g,p}(g) \,.
  \end{equation}
  Observe that (\ref{eq:pi-ps-inverse}) does \emph{not} depend on $g$:
  if
  \begin{math}
    g' = i \cdot g ,
  \end{math}
  then
  \begin{math}
    y_{g',p}
    =
    \sigma_i ( F(x_{g,p}) )
    =
    F( \sigma_i(x_{g,p}) ) ,
  \end{math}
  so that
  \begin{math}
    x_{g', p} = \sigma_i ( x_{g, p} )
  \end{math}
  by pre-injectivity, and
  \begin{math}
    x_{g', p} (g') = x_{g, p} (g) .
  \end{math}

  Let now $y$ be \emph{any} configuration asymptotic to $v$ such that
  \begin{math}
    \restrict{y}{g\Neigh}=p ,
  \end{math}
  and let $x$ be the unique pre-image of $y$ asymptotic to $u$. We
  claim that $x(g)=h(p)$.  To prove this, we observe that, as $y$ and
  $y_{g, p}$ are both asymptotic to $v$ and they agree on
  \begin{math}
    g \Neigh = g M^{-1} ,
  \end{math}
  the set $K=\Delta(y,y_{g,p})$ is finite and is contained in
  \begin{math}
    \GG \setminus g M^{-1} .
  \end{math}
  By Corollary~\ref{cor:tk}, their pre-images $x$ and $x_{g,p}$ can
  disagree only on
  \begin{math}
    K M \subseteq \left( \GG \setminus gM^{-1} \right) M .
  \end{math}
  The set $KM$ does not contain $g$, because if
  \begin{math}
    g \in \left( \GG \setminus gM^{-1} \right) M ,
  \end{math}
  then for some $m\in{M}$,
  \begin{math}
    g m^{-1} \in \left( \GG \setminus gM^{-1} \right) ,
  \end{math}
  which is not the case! Therefore,
  \begin{math}
    x(g) = x_{g,p}(g) = h(p) ,
  \end{math}
  as we claimed.

  The argument above holds whatever the pattern $p:\Neigh\to{S}$
  is. By applying it finitely many times to arbitrary finitely many
  points, we determine the following fact: if $y$ is any configuration
  which is asymptotic to $v$, then $F(H(y))=y$. But the set of
  configurations asymptotic to $v$ is dense in $S^\GG$, so it follows
  from continuity of $F$ and $H$ that $F(H(y))=y$ for every
  $y\in{S^\GG}$.

  We have thus shown that $H$ is a right inverse of $F$.  We next
  verify that $H$ is also a left inverse of $F$.

  Let $x$ be a configuration asymptotic to $u$, and set $y=F(x)$. Note
  that $y$ is asymptotic to $v$. The two configurations $x$ and $H(y)$
  are both asymptotic to $u$, and furthermore,
  \begin{math}
    F(x) = y = F(H(y)) .
  \end{math}
  Therefore, by the pre-injectivity of $F$, $x$ and $H(y)$ must
  coincide, that is, $H(F(x))=x$. The continuity of $F$ and $H$ now
  implies that the equality $H(F(x))=x$ holds even if $x$ is not
  asymptotic to $u$. Hence, $H$ is a left inverse for $F$.
\end{proof}

\begin{corollary}
  \label{cor:amen-postsurj-rev}
  A cellular automaton on an amenable group (in particular, a
  $d$-dimensional \ca) is post-surjective if and only if it is
  reversible.
\end{corollary}

\section{Post-surjectivity on sofic groups}
\label{sec:postsurj-sofic}

After proving Theorem~\ref{thm:pre-post-rev}, we might want to find a
post-surjective cellular automaton that is not pre-injective. However,
the standard examples of surjective \ca\ which are not pre-injective
fail when post-surjectivity is sought instead.  The next example
illustrates how.

\begin{example}
  \label{ex:f2-majority}
  Let $\GG=\mathbb{F}_2$ be the free group on two generators $a,b$,
  \ie, the group of \emph{reduced words} on the alphabet
  \begin{math}
    \gen = \{ a, b, a^{-1}, b^{-1} \}
  \end{math}.
  Let
  \begin{math}
    \Neigh = \gen \cup \{ 1_\GG \} = D_1
  \end{math},
  and for every
  \begin{math}
    x, y, z, w, v \in \deux
  \end{math}
  let
  \begin{math}
    f(x, y, z, w, v)
  \end{math}
  be $1$ if
  \begin{math}
    x + y + z + w + v \geq 3
  \end{math},
  and $0$ otherwise. Then
  \begin{math}
    \Acal = \langle \GG, \deux, \Neigh, f \rangle
  \end{math}
  is the \emph{majority \ca} on $\mathbb{F}_2$.

  The \ca\ $\Acal$ is clearly not pre-injective; however, it is
  surjective. Indeed, a preimage of an arbitrary pattern $p$ on $D_n$,
  for $n\geq 1$, can be obtained from a preimage of the restriction of
  $p$ to $D_{n-1}$ by exploiting the fact that every element of length
  $n$ has three neighbors of length $n+1$. We can tweak the procedure
  a little bit and see that every configuration $c$ has a (not unique)
  \emph{critical} preimage $e$ where, for every $g\in\GG$, exactly
  three between $e(g)$, $e(ga)$, $e(gb)$, $e(ga^{-1})$, and
  $e(gb^{-1})$ have value $c(g)$. An example is provided in
  Figure~\ref{fig:f2-majority}.

  Let $c$ be a configuration such that
  \begin{math}
    c(1_\GG) = c(a) = c(b) = 0
  \end{math},
  \begin{math}
    c(a^{-1}) = c(b^{-1}) = 1
  \end{math},
  and for every $n\geq 1$, each point of length $n$ has at least one
  neighbor of length $n+1$ with value $0$, and at least one neighbor
  of length $n+1$ with value $1$. Let $e$ be a critical preimage for
  $c$ which coincides with $c$ on $D_1$, and let $c'$ only differ from
  $c$ in $1_\GG$. Suppose, for the sake of contradiction, that there
  exists a preimage $e'$ of $c'$ which is asymptotic to $e$. Let $x$
  be a point of maximum length $n=\|x\|$ where $e$ and $e'$ differ.
  Call $e(x)=s$ and $e'(x)=t\neq{s}$. Two cases are possible:
  \begin{enumerate}
  \item
    $n=0$. Then $s=0$, $t=1$, and $e'(g)=e(g)$ for every
    $g\neq{1_\GG}$. But as $e$ is critical and
    \begin{math}
      c(a) = e(a) = e(1_\GG) = 0
    \end{math},
    exactly two between $e'(a^2)$, $e'(ab)$, and $e'(ab^{-1})$ have
    value $1$. As $e'(1_\GG)=1$ too, it must be $c'(a)=1$, against the
    hypothesis that $c$ and $c'$ only differ at $1_\GG$.
  \item
    $n\geq 1$. Let $u$, $v$, and $w$ be the three neighbors of $x$ of
    length $n+1$. As $e$ is critical,
  \begin{math}
    c'(u) = c'(v) = c'(w) = t
  \end{math}.
  But by construction, either $c(u)=s$, or $c(v)=s$, or $c(w)=s$. This
  contradicts that $c$ and $c'$ only differ at $1_\GG$.
  \end{enumerate}
  This proves that $\Acal$ is not post-surjective.
  \hfill\exampleqed
\end{example}

\begin{figure}
  \centering
  \begin{subfigure}[b]{0.45\textwidth}
  	\centering
    \includegraphics[width=0.9\textwidth]{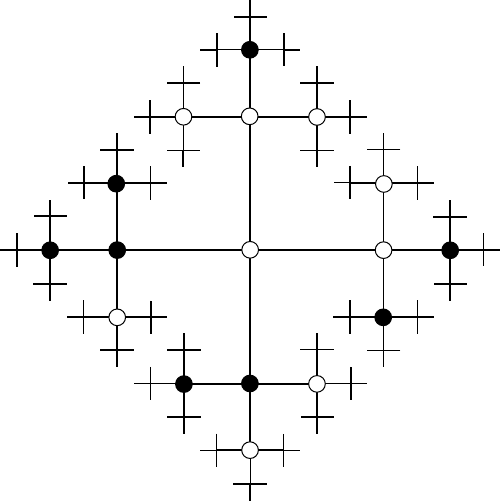}
    \caption{%
    }
    \label{subfig:f2config}
  \end{subfigure}
  \begin{subfigure}[b]{0.45\textwidth}
  	\centering
    \includegraphics[width=0.9\textwidth]{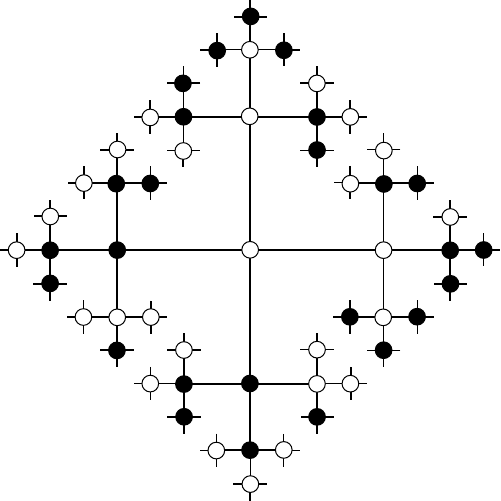}
    \caption{%
    }
    \label{subfig:f2preimg}
  \end{subfigure}
  \caption{%
    (\subref{subfig:f2config})
    A configuration on the free group on two generators, restricted to
    $D_2$.
    (\subref{subfig:f2preimg})
    A critical preimage of the configuration of
    point~\subref{subfig:f2config}, restricted to $D_3$.
  }
  \label{fig:f2-majority}
\end{figure}

The reason behind this failure is that, as we shall see below, finding
such a counterexample amounts to finding a group which is not
\emph{sofic}, and that appears to be a difficult open problem.

The notion of a sofic group was originally introduced
in~\citep{gro99}, but was later reformulated, for finitely generated
groups, in~\citep{weiss00} in combinatorial, rather than geometric,
terms.

\begin{definition}
  \label{def:sofic}
  Let $\GG$ be a finitely generated group and let $\gen$ be a finite
  symmetric set of generators for $\GG$. Let $r\geq0$ be an integer
  and $\varepsilon>0$ a real. An
  \emph{$(r,\varepsilon)$-approximation} of $\GG$ (relative to $\gen$)
  is a $\gen$-labeled graph $(V,E)$ along with a subset
  $U\subseteq{V}$ such that the following hold:
  \begin{enumerate}
  \item
    For every $u\in{U}$, the neighborhood of radius $r$ of $u$ in
    $(V,E)$ is isomorphic to $D_{\gen,r}$ as a labeled graph.
  \item
    \begin{math}
      |U| > (1 - \varepsilon) |V| .
    \end{math}
  \end{enumerate}
  The group $\GG$ is \emph{sofic} (relative to $\gen$) if for every
  choice of $r\geq 0$ and $\varepsilon>0$, there is an
  $(r,\varepsilon)$-approximation of $\GG$ (relative to $\gen$).
\end{definition}

As explained in~\citep{weiss00}, the notion of soficity does not
depend on the generating set $\gen$. For this reason, in the rest of
this section, we will suppose $\gen$ given once and for all.  It is
easy to see that finitely generated residually finite groups and
finitely generated amenable groups are all sofic.

The importance of sofic groups is threefold: firstly, as
per~\citep[Section 3]{weiss00}, sofic groups are surjunctive;
secondly, no examples of non-sofic groups are currently known. We add
a third reason:

\begin{theorem}
  \label{thm:sofic-post-pre}
  Let $\GG$ be a sofic group. Every post-surjective cellular automaton
  on $\GG$ is pre-injective (and therefore reversible).
\end{theorem}

As a corollary, cellular automata which are post-surjective, but not
pre-injective, could only exist over non-sofic groups!

To prove Theorem~\ref{thm:sofic-post-pre}, we need two auxiliary lemmas.
Observe that if $f:S^{D_R}\to S$ is the local rule of a cellular
automaton $\Acal$ on a group $\GG$ with a finite generating set
$\gen$, and $(V,E)$ is a $\gen$-labeled graph, then $f$ is applicable
in an obvious fashion to patterns on $V$ at every point $v\in V$ whose
$R$-neighborhood in $(V,E)$ is isomorphic to the disk of radius $R$ in
the Cayley graph of $\GG$ with generating set $\gen$. Therefore, we
extend our notation, and for two patterns
\begin{math}
  p : H \to S
\end{math}
and
\begin{math}
  q : C \to S
\end{math}
with
\begin{math}
  H, C \subseteq V
\end{math},
we write
\begin{math}
  p \update{f} q
\end{math}
if for every $v\in{C}$, the $R$-neighborhood $D_R(v)$ is a subset of
$H$ and is isomorphic to the disk of radius $R$, and furthermore
\begin{math}
  f \big( \restrict{p}{D_R(v)} \big) =q(v) .
\end{math}
Note that even when $\Acal$ is surjective, the induced maps
$S^H\to{S^C}$ are not necessarily surjective.

\begin{example}
  \label{ex:induced-nonsurj}
  Let $\Acal$ be the elementary \ca\ with rule 102 (same as in
  Example~\ref{ex:xor}). Let $(V,E)$ be a cycle on four nodes.
  The $1$-neighborhood of each node is isomorphic to
  $D_1\subseteq\Zset$. Let then $H=C=V$. As each bit is counted twice
  during the update (one as a center, the other as a right neighbor)
  and the rule is linear, the image in $S^C$ of an element of $S^H$
  must have an even number of $1$s. Then $0001\in{S^C}$ has no
  preimage in $S^H$.
  \hfill\exampleqed
\end{example}

\begin{lemma}
  \label{lem:postsurj-sofic}
  Let $\Acal$ be a post-surjective \ca\ on a sofic group $\GG$. Let
  $\Acal$ have state set $S$, neighborhood $\Neigh\subseteq D_R$ and local rule
  $f$, and let $N$ be given by Lemma~\ref{lem:postsurj-N}.
  Consider an $(r,\varepsilon)$-approximation given by a graph $(V,E)$
  and a set $U\subseteq V$, where $\varepsilon>0$ and $r\geq N+2R$.
  For every pattern $q:U\to S$, there is a pattern $p:V\to S$ such
  that $p\update{f}q$.
\end{lemma}

\begin{proof}
  Take arbitrary $p_0:V\to S$ and $q_0:U\to S$ such that
  $p_0\update{f}q_0$.  Let
  \begin{math}
    q_0, q_1, \ldots, q_m = q
  \end{math}
  be a sequence of patterns with support $U$ such that, for every $i$,
  $q_i$ and $q_{i+1}$ only differ in a single $k_i \in U$.
  Since the $r$-neighborhood of $k_i$ is isomorphic to the disk of the
  same radius from the Cayley graph of $\GG$, we can apply
  Lemma~\ref{lem:postsurj-N} and deduce the existence of a sequence
  \begin{math}
    p_0, p_1, \ldots, p_m
  \end{math}
  with common support $V$ such that each $p_i$ is a pre-image of $q_i$
  and, for every $i$, $p_i$ differs from $p_{i+1}$ at most in
  $D_N(k_i)$. Then $p=p_m$ satisfies the thesis.
\end{proof}
 
The next lemma is an observation made in~\citep{weiss00}.
\begin{lemma}[Packing lemma]
  \label{lem:sofic:packing}
  Let $\GG$ be a group with a finite generating set $\gen$.  Let
  $(V,E)$ be a $\gen$-labeled graph and $U\subseteq V$ a subset with
  \begin{math}
    |U| \geq \frac{1}{2} |V|
  \end{math}
  such that, for every $u\in U$, the $2\ell$-neighborhood of $u$ in
  $(V,E)$ is isomorphic to the disk of radius $2\ell$ in the Cayley
  graph of $\GG$.  Then, there is a set $W \subseteq U$ of size at
  least
  \begin{math}
    \frac{|V|}{2|D_{2\ell}|}
  \end{math}
  such that the $\ell$-neighborhoods of the elements of $W$ are disjoint.
\end{lemma}
\begin{proof}
  Let $W\subseteq U$ be a maximal set such that the
  $\ell$-neighborhoods of the elements of $W$ are disjoint.
  Then, for every $u\in U$, the neighborhood $D_{\ell}(u)$ must
  intersect the set $\bigcup_{w\in W}D_\ell(w)$.
  Therefore, $U\subseteq D_{2\ell}(W)$, which gives
  \begin{math}
    |U| \leq |D_{2\ell}| \cdot |W| .
  \end{math}
\end{proof}

\begin{proof}[of Theorem~\ref{thm:sofic-post-pre}]
  Let $\GG$ be a sofic group and assume that
  \begin{math}
    \Acal = \langle S, D_R, f \rangle
  \end{math}
  is a cellular automaton on $\GG$ that is post-surjective, but not
  pre-injective. For brevity, set $|S| = s \geq 2$. Let $N$ be as in
  Lemma~\ref{lem:postsurj-N}.

  Since the \ca\ is not pre-injective, there are two asymptotic
  configurations
  \begin{math}
    x, x' : \GG \to S
  \end{math}
  such that
  \begin{math}
    F_\Acal(x) = F_\Acal(x') .
  \end{math}
  Take $m$ such that the disk $D_m$ contains $\Delta(x,x')$. It
  follows that there are two mutually erasable patterns on $D_{m+2R}$,
  that is, two patterns
  \begin{math}
    p, p' : D_{m+2R} \to S
  \end{math}
  such that on any configuration~$z$, replacing an occurrence of $p$
  with $p'$ or vice versa does not change the image of~$z$ under
  $F_\Acal$.

  Take
  \begin{math}
    r \geq \max \{N, m\} + 2R
  \end{math}
  and $\varepsilon>0$ small. We shall need $\varepsilon$ small enough
  so that
  \begin{displaymath}
    s^{\varepsilon} \cdot \left( 1 - s^{-|D_r|} \right)^{\frac{1}{2|D_{2r}|}}
    < 1 \,.
  \end{displaymath}
  Such a choice is possible, because the second factor on the left-hand
  side is a constant smaller than $1$.
  Since $\GG$ is sofic, there is a $(2r,\varepsilon)$-approximation of
  $\GG$ given by a graph $(V,E)$ and a set $U\subseteq V$.
  Let
  \begin{math}
    \varphi : S^V \to S^U
  \end{math}
  be the map given by $\varphi(p)=q$ if $p\update{f}q$. Such $\varphi$
  is well defined, because the $R$-neighborhood of each $u\in U$ is
  isomorphic to the disk of radius $R$ in $\GG$.

  By Lemma~\ref{lem:postsurj-sofic}, the map $\varphi$ is surjective, hence
  \begin{equation} \label{eq:phi-surjective}
    |\varphi(S^V)| = s^{|U|} \,.
  \end{equation}
  On the other hand, by Lemma~\ref{lem:sofic:packing}, there is a collection
  $W\subseteq U$ of
  \begin{math}
    |W| \geq \frac{|V|}{2|D_{2r}|}
  \end{math}
  points in $U$ whose $r$-neighborhoods are disjoint. Each of these
  $r$-neighborhoods is isomorphic to the disk
  \begin{math}
    D_r \supseteq D_{m+2R}
  \end{math}
  in $\GG$. The existence of the mutually erasable patterns on $D_r$
  thus implies that there are at most
  \begin{displaymath}
    |\varphi(S^V)| \leq (s^{|D_r|}-1)^{|W|} \cdot s^{|V|-|W|\cdot|D_r|}
  \end{displaymath}
  patterns on $V$ with distinct images. However,
  \begin{eqnarray*}
    (s^{|D_r|}-1)^{|W|} \cdot s^{|V|-|W|\cdot|D_r|}
    & = &
    \left( 1 - s^{-|D_r|} \right)^{|W|} \cdot s^{|V|}
    \\ & \leq &
    \left( 1 - s^{-|D_r|} \right)^{\frac{|V|}{2|D_{2r}|}} \cdot s^{|V|}
    \\ & < &
    s^{-\varepsilon|V|} \cdot s^{|V|}
    \\ & = &
    s^{(1-\varepsilon)|V|}
    \\ & < &
    s^{|U|}
    \,,
  \end{eqnarray*}
  which contradicts (\ref{eq:phi-surjective}).
\end{proof}

\begin{corollary}
  \label{cor:sofic-postsurj}
  Let $\GG$ be a sofic group and $\Acal$ a cellular automaton on
  $\GG$. Then, $\Acal$ is post-surjective if and only if it is
  reversible.
\end{corollary}

Do post-surjective cellular automata on full shifts which are not
pre-injective exist at all? By Theorem \ref{thm:sofic-post-pre}, such
examples might exist only if non-sofic groups exist. We thus make the
following ``almost dual'' to Gottschalk's conjecture:
\begin{conjecture}
  \label{conj: sofic-postsurj}
  Let $\GG$ be a group and $\Acal$ a cellular automaton on $\GG$. If
  $\Acal$ is post-surjective, then it is pre-injective.
\end{conjecture}

\section{Balancedness}

\begin{definition}
  \label{def:balance}
  Let $\GG$ be a group and let $E\in\PF(\GG)$.  A cellular automaton
  \begin{math}
    \Acal = \langle S, \Neigh, f \rangle
  \end{math}
  on a group $\GG$ is \emph{$E$-balanced} if for every $M\in\PF(\GG)$
  such that
  \begin{math}
    E\Neigh \subseteq M ,
  \end{math}
  every pattern $p:E\to{S}$ has
  \begin{math}
    |S|^{|M|-|E|}
  \end{math}
  pre-images on $M$. $\Acal$ is \emph{balanced} if it is $E$-balanced
  for every $E\in\PF(\GG)$.
\end{definition}

If $\GG$ is finitely generated, and $r\geq0$ is such that
\begin{math}
  \Neigh \subseteq D_r ,
\end{math}
it is easy to see that Definition~\ref{def:balance} is equivalent to
the following property: for every $n\geq0$ every pattern on $D_n$
has exactly
\begin{math}
  |S|^{|D_{n+r}| - |D_n|}
\end{math}
pre-images on $D_{n+r}$. In addition (see~\citep[Remark 18]{cgk13})
balancedness is preserved by both induction and restriction: hence, it
can be determined by only checking it on the subgroup generated by the
neighborhood. Balancedness does not depend on the choice of the
neighborhood, because it is equivalent to preservation by the
\ca\ global function of the \emph{uniform product measure} on $S^\GG$
(see~\cite[Proposition 17]{cgk13}). Finally, as $|S|^{|M|-|E|}\geq1$
when $E\Neigh\subseteq{M}$, every balanced \ca\ is surjective by the
orphan pattern principle.

The notion of balancedness given in Definition~\ref{def:balance} is
meaningful for \ca\ on the full shift, but not for \ca\ on proper
subshifts. The reason is that, with proper subshifts, it may happen
that the number of patterns on a given set is not a divisor of the
number of patterns on a larger set.

\begin{example}
  \label{ex:golden-mean}
  Let
  \begin{math}
    X \subseteq \deux^\Zset
  \end{math}
  be the \emph{golden mean shift} of all and only bi-infinite
  words where the factor $11$ does not appear. It is easy to see
  (see~\citep[Example 4.1.4]{lm95}) that
  \begin{math}
    |\Lang_X \cap \deux^n| = f_{n+2} ,
  \end{math}
  where $f_n$ is the $n$th Fibonacci number. Any two consecutive
  Fibonacci numbers are relatively prime.
  \hfill\exampleqed
\end{example}

\begin{lemma}
  Let $\GG$ be a group, let $S$ be a finite set, and let
  \begin{math}
    F, H : S^\GG \to S^\GG
  \end{math}
  be \ca\ global transition functions.
  \begin{enumerate}
  \item
    If $F$ and $H$ are both balanced, then so is $F\circ{H}$.
  \item
    If $F$ and $F\circ{H}$ are both balanced, then so is $H$.
  \item
    If $H$ and $F\circ{H}$ are both balanced, and in addition $H$ is
    reversible, then $F$ is balanced.
  \end{enumerate}
\end{lemma}
In particular, a reversible \ca\ and its inverse are either both
balanced or both unbalanced.

\begin{proof}
  It is sufficient to consider the case when $\GG$ is finitely
  generated, \eg, by the union of the neighborhoods of the two \ca.
  Let $r\geq0$ be large enough that the disk $D_r$ includes the
  neighborhoods of both $F$ and~$H$.

  First, suppose $F$ and $H$ are both balanced. Let $p:D_n\to{S}$  By
  balancedness, $p$ has exactly
  \begin{math}
    |S|^{|D_{n+r}| - |D_n|}
  \end{math}
  pre-images over $D_{n+r}$ according to $H$. In turn, every such
  pre-image has
  \begin{math}
    |S|^{|D_{n+2r}| - |D_{n+r}|}
  \end{math}
  pre-images over $D_{n+2r}$ according to $F$, again by
  balancedness. All the pre-images of $p$ on $D_{n+2r}$ by $F\circ{H}$
  have this form, so $p$ has
  \begin{math}
    |S|^{|D_{n+2r}| - |D_{n}|}
  \end{math}
  pre-images on $D_{n+2r}$ according to $F\circ{H}$.  This holds for
  every $n\geq0$ and $p:D_n\to{S}$, thus, $F\circ{H}$ is balanced.

  Now, suppose $F$ is balanced but $H$ is not. Take $n\geq0$ and
  \begin{math}
    p : D_n \to S
  \end{math}
  having
  \begin{math}
    M > |S|^{|D_{n+r}| - |D_{n}|}
  \end{math}
  pre-images according to $H$. By balancedness of $F$, each of these
  $M$ pre-images has exactly
  \begin{math}
    |S|^{|D_{n+2r}| - |D_{n+r}|}
  \end{math}
  pre-images according to $F$. Then $p$ has overall
  \begin{math}
    M \cdot |S|^{|D_{n+2r}| - |D_{n+r}|} > |S|^{|D_{n+2r}| - |D_n|}
  \end{math}
  pre-images on $D_{n+2r}$ according to $F\circ{H}$, which is thus not
  balanced.

  Finally, suppose $H$ and $F\circ{H}$ are balanced and $H$ is
  reversible. As the identity \ca\ is clearly balanced, by the
  previous point (with $H$ taking the role of $F$ and $H^{-1}$ that of
  $H$) $H^{-1}$ is balanced. By the first point, as $F\circ{H}$ and
  $H^{-1}$ are both balanced, so is their composition
  \begin{math}
    F = F \circ H \circ H^{-1} .
  \end{math}
\end{proof}

As we observed after Definition~\ref{def:balance}, a balanced
\ca\ gives at least one pre-image to each pattern, thus is
surjective. On amenable groups (see~\citep{b10}) the converse is also
true; on non-amenable groups (ibid.\@) some surjective cellular
automata are not balanced. In the last section of~\citep{cgk13}, we
ask ourselves the question whether \emph{injective} cellular automata
are balanced. The answer is that, at least in all cases currently
known, it is so.

\begin{theorem}
  \label{thm:rev-bal}
  Reversible CA are balanced.
\end{theorem}
\begin{proof}
  It is not restrictive to suppose that $\GG$ is finitely generated.
  Let $\Acal$ be a reversible cellular automaton on $\GG$ with state
  set $S$ and global transition function $F = F_\Acal$.  Let $r\geq 0$
  be large enough so that the disk $D_r$ includes the neighborhoods of
  both $F$ and $F^{-1}$.  Then for every $c \in S^\GG$ the states of
  both $F(c)$ and $F^{-1}(c)$ on $D_n$ are determined by the state of
  $c$ in $D_{n+r}$.

  Let $p_1, p_2 : D_n \to S$ be two patterns. It is not restrictive to
  suppose $n \geq r$. We exploit reversibility of $F$ to prove that
  they have the same number of pre-images on $D_{n+r}$ by constructing
  a bijection $T_{1,2}$ between the set of the pre-images of $p_1$ and
  that of the pre-images of $p_2$. As this will hold whatever $n$,
  $p_1$, and $p_2$ are, $F$ will be balanced.

  For $i = 1, 2$ let $Q_i$ be the set of the pre-images of $p_i$ on
  $D_{n+r}$.  Given $q_1 \in Q_1$, and having fixed a state $0 \in S$,
  we proceed as follows:
  \begin{enumerate}
  \item
    \label{it:extend}
    First, we extend $q_1$ to a configuration $e_1$ by setting
    $e_1(g)=0$ for every $g\not\in{D_{n+r}}$.
  \item
    \label{it:direct}
    Then we apply $F$ to $e_1$ and set $c_1=F(e_1)$. By construction,
    \begin{math}
      \restrict{c_1}{D_n} = p_1
    \end{math}.
  \item
    \label{it:cut-paste}
    Next, from $c_1$ we construct $c_2$ by replacing $p_1$ with $p_2$
    inside $D_n$.
  \item
    \label{it:reverse}
    Then we set $e_2=F^{-1}(c_2)$.
  \item
    \label{it:restrict}
    Finally, we call $q_2$ the restriction of $e_2$ to $D_{n+r}$.
  \end{enumerate}
  Observe that
  \begin{math}
    q_2 = \restrict{e_2}{D_{n+r}} \in Q_2 .
  \end{math}
  This follows immediately from $\Acal$ being reversible: by
  construction, if we apply $F$ to $e_2$, and restrict the result to
  $D$, we end up with $p_2$. We call
  \begin{math}
    T_{1,2} : Q_1 \to Q_2
  \end{math}
  the function computed by performing the steps from~\ref{it:extend}
  to~\ref{it:restrict}, and
  \begin{math}
    T_{2,1} : Q_2 \to Q_1
  \end{math}
  the one obtained by the same steps with the roles of $p_1$ and $p_2$
  swapped. The procedure is illustrated in Figure~\ref{fig:rev-bal}.

  Now, by construction, $c_1$ and $c_2$ coincide outside $D_n$, and
  their updates $e_1$ and $e_2$ by $F^{-1}$ coincide outside
  $D_{n+r}$. But $e_1$ is $0$ outside $D_{n+r}$, so that updating
  $c_2$ to $e_2$ is the same as extending $q_2$ with $0$ outside
  $D_{n+r}$. This means that $T_{2,1}$ is the inverse of
  $T_{1,2}$. Consequently, $Q_1$ and $Q_2$ have the same number of
  elements. As $p_1$ and $p_2$ are arbitrary, any two patterns on
  $D_n$ have the same number of pre-images on $D_{n+r}$. As $n \geq 0$
  is also arbitrary, $\Acal$ is balanced.
\end{proof}

\begin{figure}[ht!]
  \centering

  \includegraphics[width=.90\textwidth]{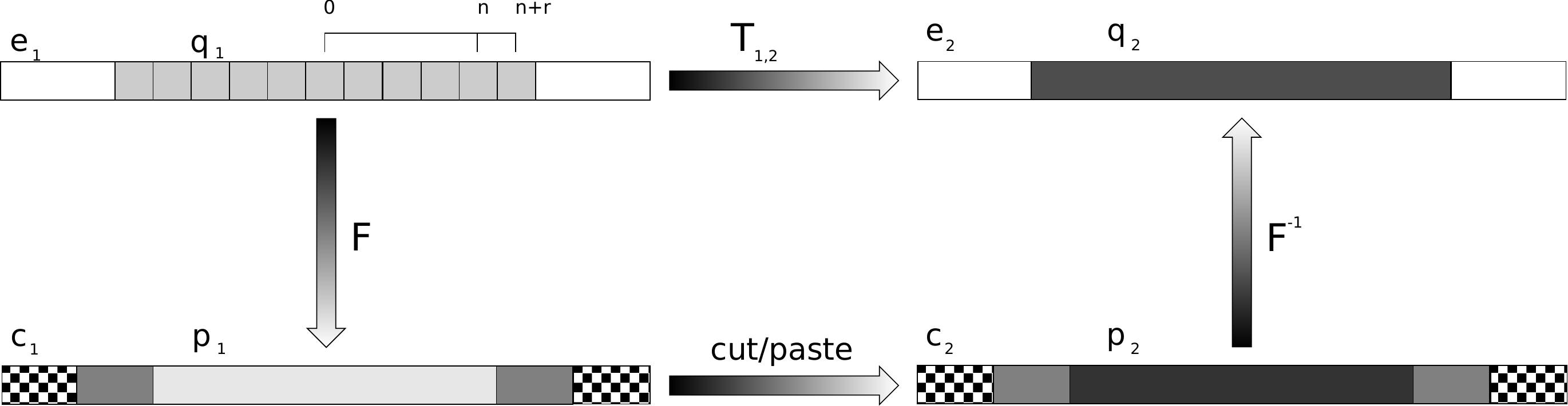}
  \caption{An illustration of the argument for Theorem \ref{thm:rev-bal}.}
  \label{fig:rev-bal}
\end{figure}

\begin{corollary}
  Injective cellular automata over surjunctive groups are balanced.
  In particular, injective \ca\ over sofic groups are balanced.
\end{corollary}

\begin{corollary}
  Gottschalk's conjecture is equivalent to the statement that every
  injective \ca\ on a full shift is balanced.
\end{corollary}
\begin{proof}
  If Gottschalk's conjecture is true, then every injective \ca\ is
  reversible, thus balanced because of Theorem~\ref{thm:rev-bal}.
  If Gottschalk's conjecture is false, then there exists a \ca\ which
  is injective, but not surjective. Such \ca\ cannot be balanced,
  because balanced \ca\ have no orphans.
\end{proof}

\nocite{*}
\bibliographystyle{abbrvnat}
\bibliography{postsurj-dmtcs}
\label{sec:biblio}

\end{document}